\documentclass[12pt]{amsart}

\usepackage{amsmath}
\usepackage{amssymb}
\usepackage{a4wide}
\usepackage{xcolor}

\newtheorem{theorem}{Theorem}[section]
\newtheorem{lemma}[theorem]{Lemma}

\newtheorem{prop}[theorem]{Proposition}
\newtheorem{cor}[theorem]{Corollary}

\newtheorem*{Theorem1'}{Theorem 1'}

\theoremstyle{definition}
\newtheorem{definition}[theorem]{Definition}

\theoremstyle{remark}

\numberwithin{equation}{section}

\newcommand \B{{\mathcal B}}
\newcommand \C{{\mathbb C}}
\renewcommand \ker{{\mathrm {ker}}}

\newcommand \Hom{{\mathrm {Hom}}}

\newcommand \GL{{\mathrm {GL}}}
\newcommand \Sp{{\mathrm {Sp}}}

\newcommand \m{{\mathfrak {m}}}
\newcommand \n{{\mathfrak {n}}}

\newcommand \la{{\lambda}}

\newcommand \si{{\sigma}}
\newcommand \ind{{\mathrm {ind}}}
\newcommand \res{{\mathrm {res}}}
\newcommand \Irr{{\mathrm {Irr}}}

\newcommand \Ba{{\mathcal B}}

\begin{document}

\title[Generators and relations for the unitary group]{Generators and relations for the unitary group of a skew hermitian form over a local ring}

\author{J. Cruickshank}
\address{School of Mathematics, Statistics and Applied Mathematics , National University of Ireland, Galway, Ireland}
\email{james.cruickshank@nuigalway.ie}

\author{F. Szechtman}
\address{Department of Mathematics and Statistics, Univeristy of Regina, Canada}
\email{fernando.szechtman@gmail.com}
\thanks{The second author was supported in part by an NSERC discovery grant}

\subjclass[2010]{20F05, 20C15, 20H25, 15B33}



\keywords{unitary group; Bruhat decomposition; group presentation; transvection}

\begin{abstract} Let $(S,*)$ be an involutive local ring and let $U(2m,S)$ be the unitary
group associated to a nondegenerate skew hermitian form defined on a free $S$-module of rank~$2m$.
A presentation of $U(2m,S)$ is given in terms of Bruhat generators and their relations. This presentation
is used to construct an explicit Weil representation of the symplectic group $Sp(2m,R)$ when $S=R$ is commutative and $*$ is the identity.

When $S$ is commutative but $*$ is arbitrary with fixed ring $R$,
an elementary proof that the special unitary group $SU(2m,S)$ is generated by unitary
transvections is given. This is used to prove that the reduction homomorphisms $SU(2m,S)\to SU(2m,\tilde{S})$ and $U(2m,S)\to U(2m,\tilde{S})$
are surjective for any factor ring $\tilde{S}$ of $S$. The corresponding results for the symplectic group $Sp(2m,R)$ are obtained as corollaries
when $*$ is the identity.
\end{abstract}

\maketitle

\section{Introduction} We are concerned with the generation, presentation and representations of the unitary group $U(2m,S)$ of rank $2m$ over an
involutive local ring $(S,*)$ associated to a standard nondegenerate skew hermitian form.

Presentations of classical groups over fields in terms of elementary matrices can be found in
\cite[Theorems 2.3.4*, 2.3.6, 2.3.8, 6.5.7, 6.5.8, 6.5.9]{HO}. For other types of presentations for these groups see \cite{ADW, Bo, E, G}. For classical groups over rings the problem is more difficult. Pantoja \cite{P} finds a presentation for the group $SSL_*(2,A)$ when $A$ is a simple
artinian ring in terms of Bruhat generators. Here $SL_*(2,A)$ is a rank 2 $*$-analogue of $SL(2,F)$, $F$ a field, isomorphic to the rank $2m$ unitary group over the division ring underlying $A$. The presentation given in \cite{P} is stated for the subgroup $SSL_*(2,A)$ of $SL_*(2,A)$
generated by the Bruhat elements and, in general, $SSL_*(2,A)$ is a proper subgroup of $SL_*(2,A)$. We follow \cite{P} and a prior
paper by Pantoja and Soto Andrade \cite{PS} in order to extend and sharpen the results of \cite{P}
by giving a Bruhat presentation of $U(2m,S)$, where $(S,*)$ is an involutive local ring, not necessarily commutative.

As an application of the above presentation we construct an
explicit Weil representation of the symplectic group $Sp(2m,R)=U(2m,R)$ when $S=R$ is
commutative and $*$ is the identity. We simply assign linear
operators to the Bruhat generators and verify that the defining
relations are satisfied. We then demonstrate that the
representation thus defined is a Weil representation, in the sense
that it is formed by intertwining operators for the
Schr$\mathrm{\ddot{o}}$dinger representation of the Heisenberg
group on which $Sp(2m,R)$ acts by means of group automorphisms. We refer the reader to \cite{GV,V},
where Weil representations of other groups of the form $SL_*(2,A)$
have also been constructed using generators and relations, and these representations were verified to be Weil by a different method,
namely by appealing to Howe's theory of reductive dual pairs. See \cite{Pr} for more on reductive dual pairs, the Weil representation and
the theta correspondence.

We also consider the generation of the special unitary group $SU(2m,S)$ by unitary transvections, where $(S,*)$ a is local, commutative
involutive ring with fixed ring $R$. The classical field case can be found in \cite{D}. The ring case, with transvections replaced
by elementary Eichler transformations, can be found in \cite[Theorem 9.2.6]{HO}. Transvections themselves were proven by Baeza \cite{B} to generate the special unitary group associated to a nondegenerate {\em hermitian} form of hyperbolic rank $\geq 1$, provided all of $R$ is the image
of the trace map $s\mapsto s+s^*$. Unlike the field case, the skew hermitian case cannot be derived from the hermitian case, since
$*$-skew hermitian units need not exist (they certainly do not exist when $(S,*)$ is ramified). We give an elementary
proof that $SU(2m,S)$ is generated by unitary transvections and use this to prove that both reduction homomorphisms $SU(2m,S)\to SU(2m,\tilde{S})$ and $U(2m,S)\to U(2m,\tilde{S})$ are surjective for any factor ring $\tilde{S}$ of $S$. The corresponding results for the symplectic group $Sp(2m,R)$ are obtained as corollaries when $*$ is the identity (see \cite{K} for the generation of $Sp(2m,R)$ by symplectic transvections).

\section{The unitary group of the standard skew hermitian form}\label{uno}

Let \( A \) be a ring  with involution
\( * \). Given an \( m \times n \) matrix \( x \) with entries in \( A \) define the \( n \times m \) matrix
\( x^* \) by \( (x^*)_{ij} = x_{ji}^* \).  This makes $*$ into an involution of the full matrix ring $M(m,A)$.
Let $V$ be a right $A$-module endowed with a skew hermitian form $h:V\times V\to A$. This means that $h$ is linear in
the second variable and satisfies
$$
h(u,v)^*=-h(v,u),\quad u,v\in V.
$$
Let $G$ be the group of all $g\in\GL(V)$ preserving $h$, in the sense that
\begin{equation}\label{defg}
h(gu,gv)=h(u,v),\quad u,v\in V.
\end{equation}
We specialize to the case when $V$ is free of rank $2m$ over $A$ and $h$ admits
$$ J = \begin{pmatrix}0 & 1_m \\ -1_m & 0
\end{pmatrix}
$$
as Gram matrix relative to a some basis $\Ba$ of $V$. Then $G$ is the unitary group $U(2m,A)$. Let $V(2m,A)$ be the subgroup
of $\GL(2m,A)$ corresponding to $U(2m,A)$ by means of $\Ba$. Given any $x\in\GL(2m,A)$, it follows from (\ref{defg}) that $x\in V(2m,A)$ if and only if
\begin{equation}\label{jota}
x^* J x=J.
\end{equation}

The group $V(2,A)$ is denoted by $SL_*(2,A)$ in \cite{PS} and \cite{P}. By (\ref{jota})
$$
x=\begin{pmatrix} a & b \\ c & d \end{pmatrix} \in \GL(2,A)
$$
is in $SL_*(2,A)$ if and only if
\begin{equation}
    \label{eqn:elements}
    a^*c = c^*a, b^*d =d^*b, d^*a - b^*c = 1.
\end{equation}
Multiplying (\ref{jota}) on the left by $(x^*)^{-1}$ and on the right by $x^{-1}$
we obtain
$$
(x^*)^{-1} J x^{-1}=J.
$$
Inverting both sides of this equation and using $J^{-1}=-J$ yields
$$
xJx^*=J.
$$
This proves that $SL_*(2,A)$ is $*$-invariant. Thus, in addition to (\ref{eqn:elements}), we have
\begin{equation}
    \label{mas}
    ab^* = ba^*, cd^* =dc^*, da^* - cb^* =1.
\end{equation}

Following \cite{PS} we define the Bruhat elements of
\( SL_*(2,A) \) by
\begin{eqnarray*}
    w & =& \begin{pmatrix} 0&1 \\ -1&0 \end{pmatrix}, \\
    h_t &=& \begin{pmatrix} t & 0 \\ 0 & (t^*)^{-1} \end{pmatrix}, t \in A^{\times},\\
    u_r &=& \begin{pmatrix} 1 & r \\ 0 & 1 \end{pmatrix}, r \in A^s. \\
\end{eqnarray*}
Here \( A^\times \) is the group of units of \( A \) and
\( A^s \) is the set of $*$-symmetric elements of \( A \). Let $D$ (resp. $N$) be the subgroup of \( SL_*(2,A) \) consisting of
all $h_t$ (resp. $u_r$) with $t\in A^\times$ (resp. $r\in A^s$). It is readily seen that $D\cap N=1$ and $DN=ND$. In particular,
$B=DN$ is a subgroup of \( SL_*(2,A) \).

Let $SSL_*(2,A)$ be the subgroup of $SL_*(2,A)$ generated by $B$ and $w$, that is,
$$
SSL_*(2,A)=B\cup BwB\cup BwBwB\cup\dots
$$

We say that \( SL_*(2,A) \) has a Bruhat decomposition if it is generated
by the Bruhat elements, that is, if $SL_*(2,A)=SSL_*(2,A)$. We say that it has a Bruhat decomposition of
length \( l \) if every element can be written as a word in the Bruhat
elements so that the word contains at most \( l \) occurrences of \( w \).

\section{A Bruhat decomposition of $SL_*(2,A)$}\label{dos}

We assume throughout this section that \( S \) is a local
ring with involution $*$ and Jacobson radical $\mathfrak j$.
The matrix ring $A=M(m,S)$ inherits an involution
from $S$, as indicated in \S\ref{uno}, also denoted by $*$.
Since $M(2m,S)\cong M(2,A)$, we have
$$
U(2m,S)\cong V(2m,S)\cong SL_*(2,A).
$$
By assumption \( \overline S = S/\mathfrak j \) is a division ring.
We extend the bar notation to elements of \( S \) and
matrices over \( S  \) in the obvious way.
That is
\( \overline s = s + \mathfrak j \) for \( s \in S \)
and for
\( x  \in A\), \( \overline x_{ij} = x_{ij}+\mathfrak j \).
We write \( \overline A \) for \( M(m,\overline S) \).

We next show that, provided the characteristic of $\overline S$ is not 2, $SL_*(2,A)$ has a Bruhat decomposition of
length \( 2 \), thereby extending the corresponding result of \cite{PS}
from division rings to local rings.

\begin{lemma}\label{aboutc} Let $x=\begin{pmatrix} a & b \\ c & d \end{pmatrix}\in SL_*(2,A)$. Then $x\in B$ if and
only if $c=0$, and $x\in BwB$ if and
only if $c\in A^\times$.
\end{lemma}

\begin{proof} Since $B=DN$ it is clear that $c=0$ provided $x\in B$. Suppose conversely that $c=0$. It follows
from (\ref{eqn:elements}) and (\ref{mas}) that $a\in A^\times$ and $d=(a^*)^{-1}$, whence $x\in DN=B$. An easy calculation reveals that entry $(2,1)$ of any element of $BwB$ is a unit. Suppose conversely that $c\in A^\times$. From (\ref{eqn:elements}) we see that $ac^{-1}\in A^s$, whence $u_{-ac^{-1}}\in N$. Since entry $(2,1)$ of $wu_{-ac^{-1}}x$ is equal to zero, we have $x\in BwB$ by the previous part.
\end{proof}

\begin{lemma}
    \label{lem:invertible}
    For \( x \in A \), \( x \in A^\times \) if and only if \( \overline x \in \overline A^\times \).
\end{lemma}

\begin{proof}
    First recall from \cite[Theorem 1.2.6]{H} that
    the Jacobson radical of \( A \) is \( M(m,\mathfrak j) \) and then use
    the fact that (for any ring \( A \))
    \( x \in A^\times \Leftrightarrow
    x+\text{rad}(A) \in A/\text{rad}(A)^\times\).
\end{proof}

\begin{lemma}
    \label{lem:coprimelemma}
    Suppose the characteristic of $\overline S$ is not 2 and that \( a,c \in A \) satisfy \( Aa+Ac = A \) and \( a^*c = c^*a \). Then
    there is some \( r \in A^s \) such that \( a+rc \in A^\times \).
\end{lemma}

\begin{proof}
    Suppose first that $\mathfrak j=0$. Let $V=S^m$, viewed as a right $S$-vector space.
    Consider the nondegenerate $*$-hermitian form $\langle~,~\rangle:V\times V\to S$, given by
$$
\langle u,v\rangle=u^*v,\quad u,v\in V.
$$
Given $x\in A\cong\mathrm{End}_S(V)$, we have
$$
\langle xu,v\rangle=\langle u,xv\rangle,\quad u,v\in V.
$$
Thus $*$ is the adjoint map associated to $\langle~,~\rangle$. We are therefore in the type I case considered
in \cite{PS} and the result follows from \cite[Proposition 3.3]{PS}.

    We next move to the general case and observe that \( \overline A \overline a + \overline A \overline c
    = \overline A\) and that \( \overline a^* \overline c = \overline c^* \overline a\).
    By above there is some \( u \in A \) such that
    \( \overline u \in \overline A^s \) and such that \( \overline a + \overline u \overline c
    \in \overline A^\times\). Let \( r = \frac12(u+u^*) \). Now
    \( r \in A^s \) and \( \overline{a+rc} = \overline{a}+\overline u\overline c
    \in \overline A^\times\). By Lemma \ref{lem:invertible},  \( a +rc \in A^\times \) as required.
\end{proof}

\begin{theorem}
    \label{thm:decomp}
    If \( A = M(m,S) \), where \( S \) is a local ring with
    involution and \( 2 \in S^\times \), then \( SL_*(2,A)
    \) has a Bruhat decomposition of length 2.
\end{theorem}

\begin{proof}
    Let $x=\begin{pmatrix} a & b \\ c & d \end{pmatrix}\in SL_*(2,A)$. Then \( a^*c = c^*a \) by (\ref{eqn:elements}),
    and $Aa+Ac=A$ because $x\in\GL(2,A)$. By Lemma \ref{lem:coprimelemma},
    we have $a+rc\in A^\times$ for some $r\in A^s$. Hence $u_{r}\in N$ and entry (2,1) of $wu_{r}x$ is a unit,
    whence $x\in BwBwB$ by Lemma \ref{aboutc}.
\end{proof}

\section{A Bruhat presentation of $SL_*(2,A)$}

We now return to the case of a general ring $A$, noticing that the Bruhat elements satisfy the following relations in $SL_*(2,A)$:
$$
h_sh_{t}=h_{st}, u_qu_{r}=u_{q+r}, w^2=h_{-1}, h_tu_r=u_{trt^*}h_t, wh_t=h_{t^{*-1}}w,
u_t w u_{t^{-1}}w u_t = w h_{-t^{-1}}.
$$
Following \cite{P}, we
define $H(A)$ to be the group with presentation
$$
\langle {z},{k}_t,{v}_r|
{k}_{s}{k}_{t} = {k}_{st}, {v}_{q}{v}_{r} = {v}_{q+r}, {z}^2 = {k}_{-1},
{k}_t{v}_r = {v}_{trt^*}{k}_t, \\ {z}{k}_t = {k}_{t^{*-1}}{z},
{v}_t {z} {v}_{t^{-1}}{z}{v}_t = {z}{k}_{-t^{-1}}\rangle.
$$

In $SL_*(2,A)$ as well as in $H(A)$, we assume that $s,t\in A^\times$ and $q,r\in A^s$, except that $t\in A^\times\cap A^s$ in the last relation.

It is clear from Theorem \ref{thm:decomp} that mapping \(  z \mapsto w,
 k_t \mapsto h_t,  v_r \mapsto u_r\) induces a group epimorphism
\( \theta: H(A) \rightarrow SSL_*(2,A)\).

Let $ C$ be the subgroup of $H(A)$ generated by all $ k_t$ and $ v_r$. By definition,
$$
H(A)= C\cup  C  z  C\cup  C  z  C  z C\cup
 C  z  C  z  C  z C\cup\dots
$$

\begin{lemma}\label{le1} Suppose $y\in H(A)$ has ${z}$ length at most 2 and $\theta(y)=1$.
Then $y=1$.
\end{lemma}

\begin{proof} If $y=k_t v_r\in C$ then $1=h_t u_r$, so $t=1$ and $r=0$, that is, $y=1$.
The case $y\in CzC$ is impossible by Lemma \ref{aboutc}, which excludes 1 from $BwB$.
Suppose finally that $y\in CzCzC$. Then $y=k_t v_a z v_b z v_c$ for some $a,b,c\in A^s$ and $t\in A^\times$.
Direct computation shows that entry $(2,1)$ of $h_t u_a w u_b w u_c$ is $(t^*)^{-1}b$, so $b=0$ and a fortiori
$h_{-t}u_{a+c}=1$, which implies $t=-1$ and $a+c=0$. This and $b=0$ yield $y=1$.
\end{proof}

\begin{lemma}
    \label{le2}
    Suppose for every \( a,b \in A^s \) such that
    \( a,b \not\in A^\times \) there exists some \( u \in A^\times \cap A^s \)
    such that \( a+u,b-u^{-1} \in A^\times \). Then
    every element of \( H(A) \)
    has \(  z \) length at most 2.
\end{lemma}

\begin{proof} We wish to show that $H(A)= C\cup  C  z  C\cup  C  z  C  z C$, that is, $C\cup  C  z  C\cup  C  z  C  z C$
is a subgroup of $H(A)$. This is equivalent to $z  C  z C z\subset  C\cup  C  z  C\cup  C  z  C  z C$,
which is a consequence of
\begin{equation}\label{cua}
 z { v_{-b}}  z {v_a}  z \in   C  z  C  z C,\quad a,b\in A^s.
\end{equation}
We proceed to prove (\ref{cua}). If either $a$ or $b$ is invertible, then (\ref{cua}) follows from the relation
\begin{equation}\label{dr}
 z { v_{t}}  z= {v_{-t^{-1}}}  z {k_{-t}} {v_{-t^{-1}}},\quad t\in A^\times\cap A^s.
\end{equation}
Suppose next neither $a$ nor $b$ is invertible. By assumption there is some $x\in A^\times \cap A^s$
    such that \( a+x,b-x^{-1} \in A^\times \). Now
$$
 z { v_{-b}}  z {v_a}  z= z { v_{-b}}  z {v_{-x}}  z  z^{-1}  {v_{a+x}} z.
$$
Since $a+x\in A^\times \cap A^s$, (\ref{dr}) guarantees that $$ z^{-1}  {v_{a+x}} z={k_{-1}} z  {v_{a+x}} z\in  C z C,$$ so it suffices to prove
$$
 z { v_{-b}}  z {v_{-x}}  z\in  C  z  C.
$$
Since $-x\in A^\times \cap A^s$, (\ref{dr}) gives
$$
 z { v_{-b}}  z {v_{-x}}  z=
 z { v_{-b+x^{-1}}}  z {k_{x}} {v_{x^{-1}}}.
$$
As $-b+x^{-1}\in A^\times \cap A^s$, (\ref{dr}) ensures that
$ z { v_{-b+x^{-1}}}  z {k_{x}} {v_{x^{-1}}}\in  C  z  C$,
as required.

\end{proof}

For the remainder of this section we adopt the assumptions and notation of \S\ref{dos}.

\begin{lemma}
    \label{le3}
    Suppose that \( |\overline S| >3 \) and that \( 2 \in S^\times \).
    Let \( a,b \in A^s \) be such that \( a,b \not\in A^\times \).
    Then there exists \( u \in A^\times \cap A^s \) such that
    \( a+u, b - u^{-1} \in A^\times \).
\end{lemma}

\begin{proof}
    By Lemma \ref{lem:invertible}, \( \overline a, \overline b
    \not\in \overline A^\times\). Now, since \( |\overline S| >3 \),
    \cite[Lemma 12]{P} ensures that there is some
    \( u \in A \) such that \( \overline u \in \overline A^\times
    \cap \overline A^s\) and such that \( \overline a
    +\overline u, \overline b - \overline u^{-1} \in \overline A^\times\).
    Replacing \( u  \) by \( \frac12(u+u^*) \), we can assume that
    \( u \in A^s \). Now Lemma \ref{lem:invertible} ensures that
    \( u \in A^\times \) and that \( a+u,b-u^{-1} \in A^\times \)
    as required.
\end{proof}

\begin{theorem}
    \label{thm:BruhatPres}
    If \( |\overline{S}| >3 \) and $2\in S^\times$ then $\theta :H(A)\to SL_*(2,A)$ is an isomorphism.
\end{theorem}

\begin{proof} By Theorem \ref{thm:decomp}, \( \theta \) is an epimorphism. Suppose $y\in H(A)$ satisfies $\theta(y)=1$.
By Lemmas \ref{le2} and \ref{le3}, $y$ has \(  w \) length at most 2, so Lemma \ref{le1} implies $y=1$.
\end{proof}

\section{A correction}

In this section we correct an error in the proof \cite[Lemma 12]{P}.
As noted in loc. cit. there are two possibilities for
the involution on \( A \) - either it is induced from the trivial involution
of \( S \) or from the Frobenius involution of \( S \). In either case, it is
shown that
\begin{equation}
    \frac{|A^\times \cap A^s|}{|A^s|} >1-\frac{q}{q^2-1}.
    \label{eqn:invertibles}
\end{equation}
The following is part of \cite[Lemma 12]{P}.
\begin{lemma}
    Suppose that \( q>3 \)
    Let \( a,b \in A^s \) be non units. There is some \( x \in A^\times \cap A^s \)
    such that \( a+x ,b - x^{-1} \in A^\times  \).
\end{lemma}

\begin{proof}
    Let \( X = (A^\times \cap A^s) - a \) and let \( Y = (A^\times \cap A^s +b)^{-1} \).
    Observe that
    \begin{equation}
        \label{eqn:Y}
        Y = \left( (A^\times \cap A^s +b) \cap (A^\times \cap A^s) \right)^{-1}.
    \end{equation}
    In particular \( Y\subset A^\times \cap A^s \). Thus to prove the result, it suffices to
    show that \( X \cap Y  \) is not empty. Now by (\ref{eqn:invertibles}),
    \begin{equation}
        \label{eqn:one}
        \frac{|X|}{|A^s|} > 1- \frac{q}{q^2-1}.
    \end{equation}

    By (\ref{eqn:Y}), \( |Y|= |Y^{-1}| = |(A^\times \cap A^s +b) \cap (A^\times \cap A^s)|\).
    Now using (\ref{eqn:invertibles}) again, together with
    the principle of inclusion and exclusion, we see that
    \begin{equation}
        \frac{|Y|}{|A^s|} > 1 - \frac{2q}{q^2-1}.
        \label{eqn:two}
    \end{equation}
    Combining (\ref{eqn:one}) and (\ref{eqn:two}) we see that
    \begin{equation*}
        \frac{|X\cap Y|}{|A^s|} > 1-\frac{3q}{q^2-1}.
    \end{equation*}
    Now for \( q>3 \), \( 1 - \frac{3q}{q^2-1} >0 \) whence the required conclusion
    follows immediately.
\end{proof}
In \cite{P} it is asserted that \( |Y| = |A^\times \cap A^s| \) which
is not necessarily true. For example, take \( q = 5 \), \( m =2  \) and \( * \) to
be induced by the trivial involution on \( \mathbb F_5 \). If
\( b = \begin{pmatrix} 1& 0 \\ 0 & 0 \end{pmatrix}\), it is clear that
\( (A^\times \cap A^s+b)^{-1} \subsetneq A^\times \cap A^s  \).
For example, \( \begin{pmatrix} 1& 4 \\ 4 & 2 \end{pmatrix} \in  A^\times  \cap A^s  \)
but \( \begin{pmatrix} 1& 4 \\ 4 & 2 \end{pmatrix}\) is not in  \((A^\times  \cap A^s+b)^{-1}  \).
Therefore
\( |(A^\times \cap A^s+b)^{-1}| < | A^\times \cap A^s | \) in this case.

\section{A representation defined by means of the Bruhat presentation}\label{ilu}

Let $R$ be a finite, commutative, local ring such that $2\in R^\times$. We assume also that $R$ has a unique minimal ideal $\n$.
This is automatic if $R$ is principal. This class of rings has been extensively investigated (see \cite{CG}, \cite{La}, \cite{Ho} and \cite{Wo}).
We write $\m$ and $F_q=R/\m$ for the maximal ideal and residue field of $R$, respective. We then have $|R|=q^d$ for a unique $d>0$.


A group homomorphism $R^+\to\C^\times$ is said to be primitive if its kernel contains no ideals of $R$ but $(0)$.
The number of non primitive group homomorphisms $R^+\to\C^\times$ is $|R/\n|$. Thus, $R^+$ has $|R|-|R/\n|>0$ primitive group homomorphisms.
For the remainder of this section and throughout the following section we fix a primitive group homomorphism $\la:R^+\to\C^\times$.
Note that for $k\in R^\times$, the group homomorphism $\la[k]:R^+\to\C^\times$, given by
$$
r\mapsto \la(kr),
$$
is also primitive. Associated to $\lambda$, we have the following quadratic Gauss sum
$$
G(\la)=\underset{r\in R}\sum\la(r^2).
$$
A brief history of these sums can be found in \cite{S}. We will appeal to two facts concerning $G(\la)$, both proven in \cite[Theorem 6.2]{S}.
To state these facts, we require the following group homomorphism $$\mu:R^\times\to \{1,-1\}\subset \C^\times,$$
defined as follows. We let $\mu$ be the trivial homomorphism if $d$ is even, and
given by
$$
\mu(r)=\begin{cases} 1 & \text{ if }r\text{ is a square in }R^\times, \\
-1 & \text{ if }r\text{ is not a square in }R^\times.
\end{cases}
$$
if $d$ is odd. We then have
\begin{equation}
\label{ga1}
G(\la)^2=\mu(-1)|R|
\end{equation}
and
\begin{equation}
\label{ga2}
G(\la[k])=\mu(k)G(\la),\quad k\in R^\times.
\end{equation}
In particular, $G(\la)\neq 0$, a fact to be used implicitly below. For an interesting use of the Weil representation
to compute the actual sign in a quadratic Gauss sum see \cite{GHH}.

Let $V$ be a free $R$-module of finite rank $2n$
and let $\langle~,~\rangle:V\times V\to R$ be an alternating form.
We assume that $\langle~,~\rangle$ is nondegenerate, in the sense that the associated linear map $V\to V^*$, given by
$v\mapsto \langle v,~\rangle$, is an isomorphism. Let $\Sp=\Sp(2n,R)$ stand for the associated symplectic group.
It is easy to see that $V$ admits a basis
$\B=\{u_1,\dots,u_n,v_1,\dots,v_n\}$ which is symplectic, in the sense that
$$
\langle u_i,v_j\rangle=\delta_{i,j},\; \langle u_i,u_j\rangle=0=\langle v_i,v_j\rangle.
$$
Let $M$ and $N$ be the span of $u_1,\dots,u_n$ and $v_1,\dots,v_n$, respectively, and consider the subgroups $\Sp_{M,N}$
and $\Sp^M$ of $\Sp$ defined as follows:
$$
\Sp_{M,N}=\{g\in\Sp\,|\, gM=M\text{ and }gN=N\},\; \Sp^M=\{g\in\Sp\,|\, gu=u\text{ for all }u\in M\}.
$$
We also consider the symplectic transformation $s\in\Sp$, defined by
$$
u_i\mapsto v_i,\; v_i\mapsto -u_i,\quad 1\leq i\leq n.
$$
Let $X$ be a complex vector space of dimension $|R|^n$ and basis indexed by $N$, say $(e_v)_{v\in N}$.

As an illustration of the use of the Bruhat presentation of $SL_*(2,A)$, we have the following result.

\begin{theorem}\label{m1} Suppose that $q>3$. Then there is one and only one group representation $W:\Sp\to\GL(X)$ such that
\begin{equation}
\label{f1}
W(g)e_v=\mu(\det g|_N)e_{gv},\quad v\in N, g\in \Sp_{M,N},
\end{equation}
\begin{equation}
\label{f2}
W(g)e_v=\la(\langle gv,v\rangle) e_{v},\quad v\in N,g\in\Sp^M,
\end{equation}
\begin{equation}
\label{f3}
W(s)e_v=G(\la)^{-n} \underset{w\in N}\sum \la\big(-2\underset{1\leq i\leq n}\sum \langle u_i,w\rangle \langle u_i,v\rangle\big) e_{w},\quad v\in N.
\end{equation}
\end{theorem}

\begin{proof} Set $A=M(n,R)$ and let $*$ be the involution of $A$ given by
$$
B\mapsto B',\text{ the transpose of }B.
$$
Fix
$$
J=\left(
    \begin{array}{cc}
      0 & 1 \\
      -1 & 0 \\
    \end{array}
  \right)\in M(2,A)
$$
and let $SL_*(2,A)$ stand for the unitary group corresponding to $J$ and~$*$, as indicated in~\S\ref{uno}.
Note that $SL_*(2,A)\cong\Sp$. Indeed, for $g\in\Sp$, let
$$M_\B(g)=\left(
            \begin{array}{cc}
              B & C \\
              D & E \\
            \end{array}
          \right)\in \GL(2,A)
$$
stand for the matrix of $g$ relative to $\B$, partitioned into 4 blocks of size $n\times n$. Obviously,
$g\mapsto M_\B(g)$ is a group monomorphism $\Sp\to \GL(2,A)$. Its image is easily seen to be $SL_*(2,A)$,
rendering a group isomorphism  $\Omega:\Sp\to SL_*(2,A)$. Under this isomorphism, $\Sp_{M,N}$ corresponds
to the subgroup of $SL_*(2,A)$ consisting of all elements
$$
h_T=\left(
    \begin{array}{cc}
      T & 0 \\
      0 & (T')^{-1} \\
    \end{array}
  \right),\quad T\in\GL(n,R),
$$
and $\Sp^M$ corresponds to the subgroup of $SL_*(2,A)$ consisting of all elements
$$
u_S=\left(
    \begin{array}{cc}
      1 & S \\
      0 & 1 \\
    \end{array}
  \right),\quad S\in M(n,R), S'=S,
$$
and $s$ corresponds to
$$
\si=\left(
    \begin{array}{cc}
      0 & -1 \\
      1 & 0 \\
    \end{array}
  \right).
$$
In view of these identifications, it suffices to prove the theorem for $SL_*(2,A)$ in terms of all $h_T$, $u_S$ and~$\si$.
Now, every $v\in N$ can be written in one and only one way in the form
$$
v=a_1v_1+\cdots+a_n v_n,
$$
where
$$
a=\left(
    \begin{array}{c}
      a_1 \\
      \vdots \\
      a_n \\
    \end{array}
  \right)\in R^n
$$
and for our matrix calculations it will be convenient
to assume that $X$ has a basis formed by all $e_a$, where $a\in R^n$. Then (\ref{f1})-(\ref{f3}) translate as follows.
\begin{equation}
\label{g1}
W(h_T)e_a=\mu(\det T)e_{(T')^{-1}a},\quad a\in R^n, T\in \GL(n,R),
\end{equation}
\begin{equation}
\label{g2}
W(u_S)e_a=\la(a'Sa) e_{a},\quad a\in R^n,S\text{ symmetric in } M(n,R),
\end{equation}
\begin{equation}
\label{g3}
W(\sigma)e_a=G(\la)^{-n}\underset{b\in R^n}\sum\la(-2b'a)e_b.
\end{equation}

Observe that
$$
\si=w^{-1}=w^2w=h_{-1}w,
$$
where $h_{-1}$ is a central involution of $SL_*(2,A)$. It follows from Theorem \ref{thm:BruhatPres} that $SL_*(2,A)$ is generated
by all $h_T$, $u_S$ and $\si$, subject to the following relations:
\begin{equation}
\label{r1}
h_{T_1}h_{T_2}=h_{T_1T_2},\quad T\in\GL(n,R),
\end{equation}
\begin{equation}
\label{r2}
u_{S_1}h_{S_2}=u_{S_1+S_2},\quad S\text{ symmetric in } M(n,R),
\end{equation}
\begin{equation}
\label{r3}
\si^2=h_{-1},
\end{equation}
\begin{equation}
\label{r4}
h_{T}u_{S}h_{T^{-1}}=u_{TST'},\quad T\in\GL(n,R),S\text{ symmetric in } M(n,R),
\end{equation}
\begin{equation}
\label{r5}
\si h_{T}=h_{(T')^{-1}}\si,\quad T\in\GL(n,R),
\end{equation}
\begin{equation}
\label{r6}
\si u_T\si=u_{-T^{-1}}\si h_T u_{-T^{-1}}.
\end{equation}
Since all $h_T$, $u_S$ and $\si$ generate $SL_*(2,A)$, it is clear that there can be at most one representation $W:SL_*(2,A)\to\GL(X)$
satisfying (\ref{g1})-(\ref{g3}). To see that such a representation exists we need to show that each linear operator appearing in (\ref{g1})-(\ref{g3}) is actually invertible and that $W$ preserves the relations (\ref{r1})-(\ref{r6}). That each $W(h_T)$ and $W(u_S)$
is an invertible linear operator of $X$ is clear. It is immediately verified that $W$ preserves (\ref{r1}), (\ref{r2}) and (\ref{r4}).

We next show that
$$
W(\si)W(\si)=W(h_{-1}),
$$
thereby proving that $W(\si)$ is invertible and that $W$ preserves (\ref{r3}). According to (\ref{g3}), we have
$$
W(\si)e_{a}=G(\la)^{-n}\underset{b\in R^n}\sum\la(-2b'a)e_b.
$$
Therefore,
$$
W(\si)W(\si)e_a=G(\la)^{-2n}\underset{b}\sum \la(-2b'a)\left( \underset{c}\sum \la(-2b'c)\right)e_{c},
$$
that is
$$
W(\si)W(\si)e_a=G(\la)^{-2n}\underset{c}\sum\underset{b}\sum \la(-2b'(a+c))e_{c}.
$$
For fixed $a,c\in R^n$, consider the linear character $\nu:R^n\to\C^\times$ given by
$$
b\mapsto \la(-2b'(a+c)).
$$
Since $\la$ is primitive, $\nu$ is trivial if and only if $a+c=0$. Thus,
$$
\underset{b}\sum  \la(-2b'(a+c))=\begin{cases} 0 & \text{ if }a+c\neq 0,\\
|R|^n & \text{ if }a+c=0.
\end{cases}
$$
It follows that
$$
W(\si)W(\si)e_a=G(\la)^{-2n}|R|^n e_{-a}.
$$
Using (\ref{ga1}), this becomes
$$
W(\si)W(\si)e_a=\mu((-1)^n) e_{-a}=W(h_{-1})e_a.
$$

We next verify that
$$
W(\si) W(h_{T})=W(h_{(T')^{-1}})W(\si),\quad T\in\GL(n,R).
$$
According to (\ref{g1}) and (\ref{g3}), we have
$$
W(\si) W(h_{T})e_a=G(\la)^{-n}\mu(\det T)\underset{b}\sum\la(-2a'T^{-1}b)e_b
$$
and
$$
\begin{aligned}
W(h_{(T')^{-1}})W(\si) e_a &=G(\la)^{-n}\mu(\det T)\underset{b}\sum\la(-2a'b)e_{Tb}\\
&=G(\la)^{-n}\mu(\det T)\underset{b}\sum\la(-2a'T^{-1}b)e_b.
\end{aligned}
$$
This proves that $W$ preserves (\ref{r5}). We finally verify that
$$
W(\si) W(u_T)W(\si)=W(u_{-T^{-1}})W(\si) W(h_T) W(u_{-T^{-1}}),\; T\text{ symmetric in }\GL(n,R).
$$
By (\ref{g3}), we have
$$
W(\si)e_a=G(\la)^{-n}\underset{b}\sum\la(-2b'a)e_b.
$$
Using first (\ref{g2}) and then (\ref{g3}) once more, we obtain
$$
W(u_T)W(\si)e_a=G(\la)^{-n}\underset{b}\sum\la(-2b'a)\la(b'Tb)e_b
$$
and
$$
W(\si)W(u_T)W(\si)e_a=G(\la)^{-2n}\underset{c}\sum\underset{b}\sum\la(-2b'a)\la(b'Tb)\la(-2b'c)e_c.
$$
We now ``complete squares", noting that
$$
\begin{aligned}
b'Tb-2b'(a+c) & =b'Tb-2b'TT^{-1}(a+c)+(a+c)'T^{-1}(a+c)-(a+c)'T^{-1}(a+c)\\
&=(b-T^{-1}(a+c))'T(b-T^{-1}(a+c))-(a+c)'T^{-1}(a+c).
\end{aligned}
$$
Since $T^{-1}$ is also symmetric, it follows that $W(\si)W(u_T)W(\si)e_a$ is equal to
$$
G(\la)^{-2n}\underset{c}\sum\underset{b}\sum\la\left([b-T^{-1}(a+c)]'T[b-T^{-1}(a+c)]\right)\la\left(-(a+c)'T^{-1}(a+c)\right)e_c.
$$
As $b$ runs through $R^n$ so does $b-T^{-1}(a+c)$ and therefore
$$
\underset{b}\sum\la([b-T^{-1}(a+c)]'T[b-T^{-1}(a+c)]=\underset{b}\sum\la(b'Tb).
$$
Now $T$ is symmetric and $q$ is odd. It follows that $T$ is congruent to a diagonal matrix $D=\mathrm{diag}(t,1,\dots,1)$,
where $t\in R^\times$. The classical field case $R=F_q$ of this result is well-known and can be found in \cite[Chapter 6]{J}.
The ring case follows from the field case by means of \cite[Theorem 2]{CQS}. Note that $r^2 t=\det(T)$ for some unit $r\in R$. Since $T$ is congruent to $D$, we infer
$$
\underset{b}\sum\la(b'Tb)=G(\la)^{n-1}G(\la[t]).
$$
Now using (\ref{ga2}) and the fact that $\mu(t)=\mu(\det T)$, we deduce that
$$
\underset{b}\sum\la([b-T^{-1}(a+c)]'T[b-T^{-1}(a+c)])=G(\la)^n \mu(\det T).
$$
All in all, we conclude that
$$
W(\si)W(u_T)W(\si)e_a=G(\la)^{-n}\mu(\det T)\underset{c}\sum\la(-(a+c)'T^{-1}(a+c)))e_c.
$$

On the other hand, according to (\ref{g2}), we have
$$
W(u_{-T^{-1}})e_a=\la(-a'T^{-1}a)e_a
$$
Using next (\ref{g1}) and then (\ref{g3}) we obtain
$$
W(h_T)W(u_{-T^{-1}})e_a=\mu(\det T)\la(-a'T^{-1}a)e_{T^{-1}a}
$$
and
$$
W(\si)W(h_T)W(u_{-T^{-1}})e_a=G(\la)^{-n}\mu(\det T)\underset{c}\sum\la(-2c' T^{-1}a)\la(-a'T^{-1}a)e_c.
$$
Applying next (\ref{g2}), we see that $W(u_{-T^{-1}})W(\si)W(h_T)W(u_{-T^{-1}})e_a$ is equal to
$$
G(\la)^{-n}\mu(\det T)\underset{c}\sum\la(-2c' T^{-1}a)\la(-a'T^{-1}a)\la(-c'T^{-1}c)e_c.
$$
Now
$$
-a'T^{-1}a-c'T^{-1}c-2c' T^{-1}a=-(a+c)'T^{-1}(a+c),
$$
so
$$
W(u_{-T^{-1}})W(\si)W(h_T)W(u_{-T^{-1}})e_a=G(\la)^{-n}\mu(\det T)\underset{c}\sum\la(-(a+c)'T^{-1}(a+c))e_c.
$$
This shows that (\ref{r6}) is preserved by $W$ and completes the proof of the theorem.
\end{proof}

\section{$W$ is a Weil representation}

Given a finite group $G$, we let $\Irr(G)$ stand for the set of all complex irreducible characters of $G$.
Given $G$-modules $P$ and $Q$, we will write
$$
[P,Q]_G=\dim_\C\Hom_{\C G}(P,Q),
$$
Given a positive integer $m$, we will write $mP$ for the $G$-module
$$
mP=P\oplus\cdots\oplus P\;(m\text{ summands}).
$$

Let $H=R\times V$ be the Heisenberg group, with multiplication
$$
(r,u)(s,v)=(r+s+\langle u,v\rangle,u+v).
$$
We identify the center $Z(H)=(R,0)$ of $H$ with $R^+$. Note that $\Sp$ acts on $H$ by means of automorphisms via
$$
{}^g (r,u)=(r,gu).
$$
Given an $R$-submodule $U$ of $V$, we consider the normal subgroup $H(U)=(R,U)$ of $H$.

\begin{prop}\label{schro} The Heisenberg group $H$ has a
unique irreducible module, up to isomorphism, such that $Z(H)$ acts on it via $\lambda$.
Its dimension is equal to $\sqrt{|V|}$.
\end{prop}

\begin{proof} Extend $\lambda$ to a linear character, say $\rho$, of $H(M)$ by means of
 $$\rho(r,u)=\lambda(r),\quad r\in R,u\in M.$$
 Since $\lambda$ is primitive, the stabilizer of $\rho$ in $H$ is easily found to be $H(M^\perp)$. As $M$ is maximal totally isotropic, $M^\perp=M$, so the stabilizer of $\rho$ in $H$ is $H(M)$ itself. Let $Y=\C y$ be a 1-dimensional complex space acted upon by $H(M)$ via $\rho$. By Clifford's irreducibility criterion,
$$
X=\ind_{H(M)}^{H(V)} Y
$$
is an irreducible $H$-module. We have
$$
\dim X=|V|/|M|=\sqrt{|V|}.
$$
By construction, $Z(H)$ acts on $X$ via $\lambda$. Viewing $Y$ as a $Z(H)$-module by restriction of scalars, this means that
$$
[Y,\res^{H(V)}_{Z(H)} X]_{Z(H)}=\sqrt{|V|}.
$$
Thus, by Frobenius reciprocity,
$$
[\ind_{Z(H)}^{H(V)} Y, X]_{H(V)} = [Y,\res^{H(V)}_{Z(H)} X]_{Z(H)}=\sqrt{|V|}.
$$
On the other hand,
$$
\dim \sqrt{|V|}X=|V|=\dim \ind_{Z(H))}^{H(V)} Y.
$$
We conclude that
$$
\ind_{Z(H)}^{H(V)} Y\cong \sqrt{|V|}X.
$$
It follows by Frobenius reciprocity that $X$ is the only irreducible $H(V)$-module, up to isomorphism, lying over $\lambda$.
\end{proof}

We fix a Schr$\ddot{\rm o}$dinger representation $S:H\to\GL(X)$ of type $\lambda$, that is,
a representation satisfying the conditions stated in Proposition \ref{schro}. For $g\in\Sp$, the conjugate representation $S^g:H\to\GL(X)$, given by
$$
S^g(r,u)=S(r,gu),
$$
is also irreducible and lies over $\lambda$. By Proposition \ref{schro}, $S$ and $S^g$ are equivalent. Thus, given any $g\in\Sp$
there is $P(g)\in\GL(X)$ such that
$$
P(g)S(h)P(g)^{-1}=S({}^g h),\quad h\in H.
$$
It follows from Schur's Lemma that the map $P:\Sp\to\GL(X)$ is a projective representation, in the sense that there is a
function $\delta:\Sp\times\Sp\to\C^\times$ such that
$$
P(g_1g_2)=\delta(g_1,g_2)P(g_1)P(g_2),\quad g_1,g_2\in\Sp.
$$
\begin{definition}\label{defw} A Weil representation of $\Sp$ of type~$\lambda$ is an ordinary representation $W:\Sp\to\GL(X)$ satisfying
$$
W(g)S(h)W(g)^{-1}=S({}^g h),\quad h\in H,g\in\Sp.
$$
\end{definition}

We next show that $W$, as defined in Theorem \ref{m1}, is a Weil representation of type $\lambda$. It should
be noted that if $W'$ is also a Weil representation of type $\lambda$, then by Schur's Lemma there is a linear character $\tau:\Sp\to\C^\times$
such that
$$
W'(g)=\tau(g)W(g),\quad g\in\Sp.
$$
But, as shown below, $\Sp$ is perfect whenever $q>3$, so in this case $W$ is the only Weil representation of type $\lambda$.

\begin{prop} Suppose $q>3$. Then $\Sp$ is perfect.
\end{prop}

\begin{proof} By Theorem \ref{symp}, $\Sp$ is generated by all transvections $f_{t,v}$, where $t\in R$ and $v\in V$ is a basis vector
(see \S\ref{despues} and \S\ref{antes} for the definitions of these objects). Thus, it is enough to show that given any basis
vector $v$ and any $t\in R$, the transvection $f_{t,v}\in\Sp'$. For this purpose, note that for any unit $s\in R$
the vector $sv$ is also a basis vector. It is easy to see that any basis vector is part of a symplectic basis of $V$.
Thus, $\Sp$ acts transitively on basis vectors. Hence, there is $g\in\Sp$ such that
$gv=sv$. Therefore, for any $r\in R$, we have
$$
[g,f_{r,v}]=gf_{r,v}g^{-1}f_{-r,v}=f_{r,gv}f_{-r,v}=f_{r,sv}f_{-r,v}=f_{s^2r,v}f_{r,v}=f_{(s^2-1)r,v}.
$$
Choose the unit $s$ so that $s^2\not\equiv 1\mod\m$, which is possible since $q>3$. Then $s^2-1\in R^\times$ and we may set
$r=(s^2-1)^{-1}t$. This proves that $f_{t,v}\in\Sp'$.
\end{proof}

\begin{theorem}\label{m2} Suppose $q>3$. Then the representation $W:\Sp\to\GL(X)$ of Theorem \ref{m1} is a Weil representation of type $\lambda$
in the sense of Definition \ref{defw}.
\end{theorem}

\begin{proof} In order to verify that $W$, as defined in Theorem \ref{m1}, satisfies Definition \ref{defw}, we first determine the action of $H(V)$ on $X$.
Note that the elements $(0,v)$, $v\in N$, of $H(V)$ form a transversal for $H(M)$ in $H(V)$,
so that
$$
X=\underset{v\in N}\bigoplus (0,v)\C y.
$$
It follows that the elements $e_v=(0,v)y$, $v\in N$, form a complex basis of $X$. Moreover,
we have the following formulae for the action of $H(V)$ on basis vectors $e_v$, $v\in N$:
\begin{equation}
\label{a1}
S(0,w)e_v=e_{v+w},\quad w\in N,
\end{equation}
\begin{equation}
\label{a2}
S(0,u)e_v=\lambda(2 \langle u,v\rangle )e_{v},\quad u\in M,
\end{equation}
\begin{equation}
\label{a3}
S(r,0)e_v=\lambda(r)e_{v},\quad r\in R.
\end{equation}
As $(0,M)\cup(0,N)$ generate $H(V)$, given $g\in\Sp$, it suffices to verify that $W(g)$ satisfies Definition \ref{defw} for $h$ in $(0,M)$
or $(0,N)$. Since $W$, as given in Theorem~\ref{m1}, was proven to a be a representation, we may restrict this verification to $g$ in the generating set of $\Sp$ used in the proof of Theorem \ref{m1}, namely $\Sp_{M,N}\cup\Sp^M\cup\{s\}$.

Suppose first $g\in\Sp_{M,N}$. Then for $u\in M$ and $v\in N$, we have
$$
\begin{aligned}
W(g)S(0,u)W(g)^{-1}e_v &=W(g)S(0,u)\mu(\det g^{-1}|_N)e_{g^{-1}v}\\
&=W(g)\mu(\det g^{-1}|_N)\lambda(2\langle u,g^{-1}v\rangle )e_{g^{-1}v}\\
&=\mu(\det g|_N)\mu(\det g^{-1}|_N)\lambda(2\langle u,g^{-1}v\rangle )e_{gg^{-1}v}\\
&=\lambda(2\langle gu,v\rangle )e_{v}=S^g(0,u)e_v.
\end{aligned}
$$
Moreover, for $w\in N$ and $v\in N$, we have
$$
\begin{aligned}
W(g)S(0,w)W(g)^{-1}e_v &=W(g)S(0,w)\mu(\det g^{-1}|_N)e_{g^{-1}v}\\
&=W(g)\mu(\det g^{-1}|_N)e_{g^{-1}v+w}\\
&=\mu(\det g|_N)\mu(\det g^{-1}|_N)e_{v+gw}\\
&=e_{v+gw}=S^g(0,w)e_v.
\end{aligned}
$$

Suppose next $g\in\Sp^M$. Then for $u\in M$ and $v\in N$, we have
$$
\begin{aligned}
W(g)S(0,u)W(g)^{-1}e_v &=W(g)S(0,u)\la(\langle g^{-1}v,v\rangle )e_{v}\\
&=W(g)\lambda(2\langle u,v\rangle)\la(\langle g^{-1}v,v\rangle ) e_{v}\\
&=\lambda(2\langle u,v\rangle)\la(\langle g^{-1}v,v\rangle ) e_{v}\\
&=\lambda(2\langle u,v\rangle )\la(\langle g^{-1}v,v\rangle )\la(\langle gv,v\rangle )e_{v}\\
&=\lambda(2\langle u,v\rangle )e_{v}=S(0,u)e_v=S^g(0,u)e_v.
\end{aligned}
$$
We used above the fact that $g\mapsto \la(\langle gv,v\rangle )$ is a group homomorphism $\Sp^M\to\C^\times$,
which appeared in \S\ref{ilu} in matrix form as the group homomorphism $u_S\mapsto \la(a'Sa)$.
Moreover, for $w\in N$ and $v\in N$, we have
$$
\begin{aligned}
W(g)S(0,w)W(g)^{-1}e_v &=W(g)S(0,w)\la(\langle g^{-1}v,v\rangle )e_{v}\\
&=W(g)\la(\langle g^{-1}v,v\rangle ) e_{v+w}\\
&=\lambda(\langle g(v+w),v+w\rangle)\la(\langle g^{-1}v,v\rangle ) e_{v+w}\\
&=\lambda(\langle gv,w\rangle)\lambda(\langle gw,v\rangle)\lambda(\langle gw,w\rangle) e_{v+w}.
\end{aligned}
$$
On the other hand, since $gw-w\in M$ and
$$
(0,gw)=(0,gw-w+w)=(0,gw-w)(0,w)(\langle w,gw\rangle,0),
$$
we have
$$
\begin{aligned}
S^g(0,w)e_v &= S(0,gw)e_v\\
&=S(0,gw-w)S(0,w)S(\langle w,gw\rangle,0)e_v\\
&=S(0,gw-w)S(0,w)\la(\langle w,gw\rangle)e_v\\
&=S(0,gw-w)\la(\langle w,gw\rangle)e_{v+w}\\
&=\lambda(2\langle gw-w,v+w\rangle)\la(\langle w,gw\rangle)e_{v+w}\\
&=\lambda(\langle gw,w\rangle)\lambda(2\langle gw,v\rangle)e_{v+w}\\
&=\lambda(\langle gw,w\rangle)\lambda(\langle gw,v\rangle)\lambda(\langle gw,v\rangle)e_{v+w}.
\end{aligned}
$$
This is equal to the above expression, since
$$
\langle gw,v\rangle=\langle gv,w\rangle,
$$
which is most easily seen in matrix form as
$$
a'Sb=b'Sa
$$
for any symmetric matrix $S\in M(n,R)$. We finally show that
$$
W(s)S(h)W(s)^{-1}=S({}^s h),\quad h\in H,
$$
or equivalently,
$$
W(s)S(h)=S({}^s h)W(s),\quad h\in H.
$$
For this purpose, let
$$
v=a_1v_1+\cdots+a_n v_n,\, w=b_1v_1+\cdots+b_n v_n,\,z=c_1v_1+\cdots+c_n v_n,\, u=d_1u_1+\cdots+d_n u_n.
$$
By definition of $s$, we have
$$
s(w)=-(b_1u_1+\cdots+b_n u_n),\, s(u)=d_1v_1+\cdots+d_n v_n.
$$
Therefore, for any $a\in R^n$, we have
$$
\begin{aligned}
S(0,s w)W(s)e_{a} &= S(0,s w)G(\la)^{-n}\underset{c}\sum \la(-2c'a)e_{c}\\
&=G(\la)^{-n}\underset{c}\sum \la(-2c'a)\la(-2c'b)e_{c}
\end{aligned}
$$
and
$$
\begin{aligned}
W(s)S(0,w)e_{a} &= W(s)e_{a+b}\\
&=G(\la)^{-n}\underset{c}\sum \la(-2c'(a+b))e_{c},
\end{aligned}
$$
as required. Moreover,
$$
\begin{aligned}
S(0,s u)W(s)e_{a} &= S(0,s u)G(\la)^{-n}\underset{c}\sum \la(-2a'c))e_{c}\\
&=G(\la)^{-n}\underset{c}\sum \la(-2a'c))e_{c+d}\\
&=G(\la)^{-n}\underset{c}\sum \la(-2a'(c-d)))e_{c}
\end{aligned}
$$
and
$$
\begin{aligned}
W(s)S(0,u)e_{a} &= W(s)\la(2a'd)e_a\\
&=G(\la)^{-n}\underset{c}\sum \la(2a'd) \la(-2a'c) e_{c}\\
&=G(\la)^{-n}\underset{c}\sum \la(-2a'(c-d)))e_{c},
\end{aligned}
$$
as required. This completes the proof that $W$ is a Weil representation of type $\la$.
\end{proof}

\section{Transvections}\label{despues}

In this section we assume that \( A \) is an arbitrary commutative
ring. Let $V$ be a free $A$-module of finite rank $n$. Given an $A$-linear map \( \theta:V \rightarrow A \) and
some \( v \in \ker(\theta) \) we define the linear transformation \( t_{\theta,v}:V \rightarrow V \)
by
\[ t_{\theta,v}(u) = u +\theta(u)v. \]
We call \( t_{\theta,v} \) a {\em transvection} of \( V \).
We refer to a vector \( u \) of \( V \) as a basis vector if it
belongs to some basis of \( V \).

\begin{lemma}
    \label{lem:transvdetone}
    \( \det(t_{\theta,v}) =1 \).
\end{lemma}

\begin{proof} Let $\B=\{v_1,\dots,v_n\}$ be a basis of $V$. We embed $V$ into a free $A$-module $\tilde{V}$ of rank $n+1$
with basis ${\mathcal C}=\{v_0,v_1,\dots,v_n\}$ and define the linear functional
\( \tilde{\theta}: \tilde{V} \rightarrow A \)
    by \( \tilde{\theta}(av_0+u) =\theta(u)\) for \( a \in A \) and
    \( u \in V \). Let \( \tilde{v} = v_0+v\) and observe that
    \( \tilde{v} \) is a basis vector of $\tilde{V}$ and that the matrix of \(
    t_{\tilde\theta,\tilde{v}}\) with respect to ${\mathcal C}$
    is \( \begin{pmatrix} 1 & y \\ 0 & X \end{pmatrix} \)
    where \( X \) is the matrix of \( t_{\theta,v} \) with respect to $\B$.
    Since \( \det(t_{\tilde\theta,\tilde{v}}) = \det(t_{\theta,v}) \),
    the result follows from the case where \( v \) is assumed to be
    a basis vector of $V$.
    In that case we may assume, without loss of generality, that
    \( v=v_1\). Now the matrix
    of \( t_{\theta,v}  \) with respect to $\B$ is
    \[ \begin{pmatrix}
        1 & c_2 & c_3 & \cdots & c_l \\
        0 & 1 & 0 & \cdots & 0 \\
        0 & 0 & 1 & \cdots & 0 \\
        \vdots &&&& \vdots \\
        0 &0 &0 & \cdots & 1
    \end{pmatrix}
    \]
    where \( c_i = \theta(v_i)\) for \( i = 2,\cdots,n \), so \( \det(t_{\theta,v}) =1 \) as claimed.

\end{proof}

\begin{proof}[Alternative proof]
Let \(v_1,\dots,v_n\) be a basis of \(V\) and write
\(v=a_1v_1+\dots+a_nv_n\).
Also let \(\theta(v_1)=b_1,\dots,\theta(v_n)=b_n\). The matrix of
\(t_{\theta,u}\)  with respect to this basis is
\[
    \begin{pmatrix}
        1+a_1b_1 & a_1b_2 & \dots & a_1b_n \\
        a_2b_1 & 1+ a_2b_2 & \dots & a_2b_n \\
        \vdots & & & \vdots \\
        a_nb_1 & a_nb_2 & \dots & 1 +a_nb_n
    \end{pmatrix}
\]
The determinant of this matrix is \(1+a_1b_1+\dots+a_nb_n\)
(which is a nice exercise in induction for the reader!).
But \(0=\theta(v)=a_1b_1+\dots+a_nb_n\).
\end{proof}

\section{Generation of the special unitary group by transvections}\label{antes}

Throughout this section \( S \) stands for a commutative local ring with involution \( * \)
and \( R \) for the subring of elements fixed by \( * \).
Let \( V \) be a free \( S \) module of rank \( 2m \) and
let \( h:V \times V \rightarrow S \) be a nondegenerate skew hermitian
form on V. We say that elements \( u_1,v_,\dots,u_l,v_l \in V\) form
a symplectic set if
\( h(u_i,v_i) = 1\) for all \(i\), \( h(u_i,u_j) =
h(v_i,v_j) = 0
\) for all \( i,j \), and \( h(u_i,v_j) = 0 \) for all \( i \neq j \).
We will assume that \( V \) has a symplectic basis
\( u_1,v_1,\dots,u_m,v_m \). Note that in general, given \( V \) and
\( h \) there may not exist such a symplectic basis. Indeed, in
\cite{CQS} we have explored sufficient conditions for the existence
of such a basis. Here we will just assume that such a basis exists
for \( V \). Note that we do not assume the existence of such a basis
for arbitrary free submodules of \( V \).

We write \( U(V) \) for the group
of \( S \)-linear automorphisms that preserve \( h \), and \( SU(V) \)
for the subgroup consisting of elements of determinant 1.
In the sequel we will, where convenient and without further comment,
identify endomorphisms
of \( V \) with their matrix representations with respect to
the basis \( u_1,v_1,\dots,u_m,v_m \).


Given \( v \in V \) we call \( h(v,v) \) the {\em length} of
\( v \).
For \( a \in S \) and \( v \in V \), define
\[ f_{a,v}(u) = u + ah(v,u)v. \]
Observe that if \( v \) has length zero, then \( f_{a,v} \)
is a transvection.

\begin{lemma}
    \label{lem:unitarytransevections}
    If \( a \in R \) and \( h(v,v) = 0 \)
    then \( f_{a,v}
    \in SU(V)\).
\end{lemma}

\begin{proof}
    Observe that
    \begin{eqnarray*}
        h(f_{a,v}(w_1),f_{a,v}(w_2)) &=& h(w_1+a h(v,w_1)v,w_2+a h(v,w_2)v)
    \end{eqnarray*}
    Now, expanding the right hand side and using \( h(v,v) =0 \) and
    \( a^*=a \), we see that \( f_{a,v} \in U(V) \). As observed above,
    \( f_{a,v} \) is a transvection and hence, by Lemma \ref{lem:transvdetone},
    it has determinant 1.
\end{proof}

The goal of this section is to prove the following theorem.

\begin{theorem}
    \label{thm:SUgenbytransvections}
    The set \[ \{f_{a,v}: a \in R, h(v,v) = 0, v \text{ a basis vector}\} \] is a generating
    set for the group \( SU(V) \).
\end{theorem}

\begin{lemma}\label{p}
    The following are equivalent.
    \begin{enumerate}
        \item \( u \) is a basis vector of \( V \).
        \item There is some \( w \) belonging to the
            basis \( u_1,v_1,\dots,u_m,v_m \) such that
            \( h(u,w) \in S^\times \).
        \item There is some \( v \in V \) such that \( h(u,v)
            \in S^\times\).
        \item There is some \( z \in V \) such that \( h(u,z) =1 \).
    \end{enumerate}
    \label{lem:basisone}
\end{lemma}

\begin{proof} (1) \( \Rightarrow  \) (2) We have
    \( u = \sum_{i=1}^m (a_iu_i+b_iv_i) \) where \( b_i = h(u_i,u) \)
    and \( a_i = -h(v_i,u) \). Since $u$ is a basis vector, $(a_1,b_1,\dots,a_m,b_m)$
    is a row of some invertible matrix over $S$. As \( S \) is local this implies that
    at least one entry in the row is a unit.

    (2) \( \Rightarrow  \) (3) is obvious, while (3) \( \Rightarrow  \) (4) follows from
    \( 1 = h(u,h(u,v)^{-1}v) \) if \( h(u,v) \in S^\times \).

    (4) \( \Rightarrow  \) (1) We have
    \( u = \sum_{i=1}^m (a_iu_i+b_iv_i) \) for some $a_i,b_i\in S$. Since $S$ is local, if all $a_i,b_i$ are non units
    then so is $h(u,z)$ for every $z\in V$. As this is not the case, some $a_i$ or $b_i$ is a unit, whence $u$ is a basis vector.
\end{proof}

Let \( T_{2m} \) be the subgroup of \( SU(V) \) generated
by \( \{f_{a,v}: a \in R, h(v,v) = 0, v \text{ a basis vector}\} \).
Let \( \sim \) be the equivalence relation on \( V \) induced
by the action of \( T_{2m} \), so \( u \sim v \) if and only
if there is some \( X \in T_{2m} \) such that \( Xu=v \).
Similarly for symplectic pairs
\( (u,v) \) and \( (w,z) \) we write \( (u,v)\sim (w,z) \) if there
is some \( X \in T_{2m} \) such that \( Xu = w \) and \( Xv = z \).
The key step in the proof of Theorem \ref{thm:SUgenbytransvections}
is to show that for \( m \geq 2 \), \( T_{2m} \) acts
transitively on the set of symplectic pairs in \( V \).

\begin{lemma}
    \label{lem:prodinRimpliesrelated}
    Let \( u,v \) be basis vectors of \( V \), both of length \( 0 \),
    and suppose that \( h(u,v) \in R^\times \).
    Then \( u \sim v \).
\end{lemma}

\begin{proof}
    We observe that \( h(v-u,v) = -h(u,v) \in R^\times \).
    Therefore, by Lemma \ref{lem:basisone},
    \( v-u \) is a basis vector. Moreover \( v-u \) has length zero since \( h(u,v) \in R \).
    So
    \( f_{-a^{-1},v-u} \in T_{2m} \). One readily checks that
    \( f_{-a^{-1},v-u}(u) = v \).
\end{proof}

\begin{lemma}
    \label{lem:UrelatedtoaU}
    Let \( u \) be a basis vector of length zero and let \( a \in R^\times \).
    Then \( u \sim au \).
\end{lemma}

\begin{proof}
    By Lemma \ref{lem:basisone} there is some \( v' \) such that \(
    h(u,v')= 1\). Let \( v = v'+\frac12 h(v',v')u \).
    So \( (u,v) \) is a symplectic pair.
    By Lemma \ref{lem:prodinRimpliesrelated},
    \( u \sim u +a^{-1}v \). However, \( (au, a^{-1}v)\)
    is also a symplectic pair. Therefore \( au \sim a^{-1}v \sim a^{-1}v + u \),
    also by Lemma \ref{lem:prodinRimpliesrelated}.
\end{proof}


\begin{lemma}
    Suppose that \( u \) and \( v \) belong to a common symplectic
    set. Then \( u \sim v \).
    \label{lem:sympbasis}
\end{lemma}

\begin{proof}
    Suppose that \( w_1,z_1,\dots,w_l,z_l \) is a symplectic set.
    By
    Lemma \ref{lem:prodinRimpliesrelated},
    \( w_i \sim z_i \) for \( i = 1,\cdots,m \).
    So it suffices to show that \( w_i \sim w_j \) for \( i \neq j \).
    But \( w_i \sim z_i+z_j \) by Lemma~\ref{lem:prodinRimpliesrelated} and
    similarly \( w_j \sim z_i+z_j \).
\end{proof}

\begin{lemma}
    Let \( u \) be a basis vector of length zero. If \( m \geq 2 \) then
    \( u \sim u_1 \).
\end{lemma}

\begin{proof}
    As above
    \( u = \sum_{i=1}^m (a_i u_i + b_iv_i) \) where \( a_i = h(u,v_i ) \)
    and \( b_i = -h(u,u_i) \).
    By Lemma~\ref{lem:basisone} we may suppose,
    without loss of generality, that \( a_1 \in S^\times \). Then
    \[ a_1^*u_1,a_1^{-1}v_1,u_2,v_2,\cdots,u_m,v_m \] is a symplectic
    basis of \( V \) and \( h(u,a_1^{-1}v_1) = 1 \in R \).
    By Lemma \ref{lem:prodinRimpliesrelated}, \( u \sim a_1^{-1}v_1 \).
    But, by Lemma \ref{lem:sympbasis} and since \( m \geq 2 \),
    \( a_1^{-1}v_1 \sim u_2 \). Therefore \( u \sim u_2 \).
    By Lemma \ref{lem:sympbasis} again, \( u_2 \sim u_1 \).
\end{proof}

As an immediate consequence we have the following.

\begin{cor}
    \label{cor:transitivityonbasisvectorsoflengthzero}
    If \( m \geq 2 \) then \( T_{2m} \) acts transitively on
    the set of basis vectors of length zero.
\end{cor}

\begin{lemma}
    \label{lem:symppairprodunitinR}
    Suppose that \( (u,v) \) and \( (u,w) \) are symplectic pairs
    and that \( h(v,w) \in R^\times \). Then \( (u,v) \sim (u,w) \)
\end{lemma}

\begin{proof}
    Observe that \( f_{h(w,v)^{-1},w-v} \), fixes
    \( u \) and maps \( v \) to \( w \).
\end{proof}

\begin{lemma}
    Suppose that \( m \geq 2 \) and that \( (u,v) \) is a symplectic pair.
    We then have \( (u,v) \sim (u_1,v_1) \).
\end{lemma}

\begin{proof}
    By Corollary
    \ref{cor:transitivityonbasisvectorsoflengthzero} we may assume
    without loss of generality that \( u = u_1 \).
    Since \( h(u_1,v) =1 \) we can write \( v = a_1u_1 + v_1 +
    a_2u_2+b_2 v_2 + \cdots + a_m u_m+b_m v_m\).
    Now suppose that one of \( a_2,b_2,\cdots,a_m,b_m \) is a unit.
    Say that \( a_2 \) is a unit --similar arguments
    apply to all other cases. For \( \delta \in S \), let \( z_\delta
    = u_1+v_1+\delta v_2\). Now \( (u_1,z_\delta) \) is a symplectic pair
    and \( h(z_\delta,v_1) = 1 \). Therefore \( (u_1,v_1)\sim (u_1,z_\delta) \)
    for all \( \delta \in S \)
    by Lemma \ref{lem:symppairprodunitinR}. On the other
    hand \( h(z_\delta,v) = 1-a_1-a_2\delta^* \). Since \( a_2  \) is
    a unit we see that \( (u_1,z_\delta) \sim (u_1,v) \) if
    \( \delta = -(a_2^{-1}a_1)^* \). Therefore \( (u_1,v_1) \sim (u_1,v) \)
    in this case.

    Now suppose that none of \( a_2,b_2,\cdots,a_m,b_m \) is a unit.
    In this case we observe that \( u_1,v_1+u_2,u_2,v_2+u_1,\cdots,u_m,v_m \)
    is also a symplectic basis of \( V \). Now we see that
    \( v = (a_1-b_2)u_1 +v_1+u_2 +(-1+a_2)u_2 +b_2(v_2+u_1) +
    \sum_{i\geq 3} a_i u_i +
    b_i v_i\). But \( -1+a_2 \in S^{\times} \) since \( S \) is local so,
    by the previous paragraph, \( (u_1,v_1+u_2) \sim (u_1,v) \).
    Finally, two applications of Lemma \ref{lem:symppairprodunitinR}
    yield
    \( (u_1,v_1) \sim (u_1,u_1+v_1) \sim (u_1,v_1+u_2) \). Thus
    \( (u_1,v_1) \sim (u_1,v) \) in this case as well.
\end{proof}

\begin{cor}
    \label{cor:transitivityonsymppairs}
    If \( m \geq 2 \) then \( T_{2m} \) acts transitively on the
    set of symplectic pairs in \( V \).
\end{cor}

Now we turn to the special case \( m=1 \).
Thus, let \( W = S^2 \) and define \( h: W \times W
\rightarrow S\) by \( h(u,v) = u^*Jv \) where
\( J  = \begin{pmatrix} 0&1 \\ -1 & 0 \end{pmatrix} \).
Let \( u_1,v_1 \) be the standard basis of \( W \).
For $b,c\in R$,
\( U_b = \begin{pmatrix} 1 & b  \\ 0 & 1 \end{pmatrix} \)
is the matrix of \( f_{b,u_1} \) and
\( L_c = \begin{pmatrix} 1 & 0  \\ c & 1 \end{pmatrix} \)
is the matrix of \( f_{-c,v_1} \). We next identify $U(V)$
with its matrix representation relative to the basis \( u_1,v_1 \).
Thus \( U_b,L_c \in T_2 \)
for all \( b,c \in R \).
Moreover, \( w = \begin{pmatrix} 0& 1 \\ -1 & 0 \end{pmatrix} \)
can be written as \( U_1L_{-1}U_1 \) and \( -I_2 = w^2 \), so these
matrices also belong to \( T_2 \).

\begin{lemma}
    \label{lem:su2sl2}
    Under the matrix representation described above
    \( SU(W) = SL(2,R)\).
\end{lemma}

\begin{proof}
    Let \( X = \begin{pmatrix} a&b \\ c&d \end{pmatrix} \) be an
    element of \( GL(W) \). Then $X\in SU(W)$ if and only if
\begin{equation}
\label{cond}
 a^*c \in R, b^*d \in R, a^*d - c^*b = 1, ad -bc =1.
\end{equation}
Suppose first that these conditions hold. Then, since \( S \) is commutative, we have
    \begin{eqnarray*}
        a^*cb &=& c^*ab \\
        &=& a(a^*d-1) \\
        &=& a^*ad-a \\
        &=& a^*(1+bc)-a\\
        &=& a^*cb+a^*-a
    \end{eqnarray*}
    whence \( a^* = a \).
    Similar arguments yield \( b,c,d \in R \). So \( X \in SL(2,R) \).
    On the other hand, it is clear that if \( X \in SL(2,R) \) then all conditions (\ref{cond}) hold.
\end{proof}

Let $B(2,R)$ be the subgroup of $SL(2,R)$ of upper triangular matrices.

\begin{lemma}
    \label{lem:diagonals} $B(2,R)$ is included in $T_2$.
    \end{lemma}

\begin{proof} Let $Y=\begin{pmatrix} a & b \\0 & a^{-1}
    \end{pmatrix}\in B(2,R)$. By Lemma \ref{lem:UrelatedtoaU}, there is some \( X \in T_2 \)
    such that \( Xu_1 = au_1 \). Therefore \( X =
    \begin{pmatrix} a & c \\ 0 & a^{-1} \end{pmatrix}\), since \( \det(X) = 1 \). Thus $Y=XU_d$ for a suitable $d\in R$,
    whence $Y\in T_2$.
\end{proof}

\begin{lemma}
    \label{lem:t2sl2}
    \( T_2 = SL(2,R) \).
\end{lemma}

\begin{proof}
    Suppose that \( X = \begin{pmatrix}a&b\\c&d \end{pmatrix} \in
    SL(2,R)\). Replacing \( X \) by \( wX \) if necessary, we can
    assume that \( a \in R^\times \).
    Now \( L_{-ca^{-1}}X\in B(2,R) \), so $X\in T_2$ by Lemma
    \ref{lem:diagonals}.
\end{proof}

Finally, we are in position to prove Theorem
\ref{thm:SUgenbytransvections}.

\begin{proof}[Proof of Theorem \ref{thm:SUgenbytransvections}]
    We prove this by induction on \( m \) where \( 2m \) is the
    rank of \( V \). The case \( m =1  \) is given by Lemmas
    \ref{lem:su2sl2} and \ref{lem:t2sl2}.
    Suppose that \( m \geq 2 \) and let \( X \in SU(V) \).
    By Corollary \ref{cor:transitivityonsymppairs}
    there is some \( Y \in T_{2m} \) such that
    \( Yu_m = Xu_m \) and \( Yv_m = Xv_m \).
    Now let \( V' =  Su_1+Sv_1 + \dots +Su_{m-1} +Sv_{m-1}  \).
    Observe that \( V' = (Su_m+Sv_m)^\perp \).
    Since \( Y^{-1}X \) fixes both
    \( u_m \) and \( v_m \), we see that
    \( V' \) is invariant under \( Y^{-1}X \). So
    \( Y^{-1}X   = Z \oplus I_2\) where \( Z \in GL(V') \).
    Moreover it is clear that \( Z \in SU(V') \) since
    \( Y^{-1}X \in SU(V) \).
    By induction, \( Z \in T_{2(m-1)} \) and so
    \( X = Y(Z \oplus I_2) \in T_{2m} \) as required.
\end{proof}

We next specialize to the case when the involution $*$ is trivial. In this case, $S=R$ and $h:V\times V\to R$ is
a nondegenerate alternating bilinear form. The subgroup of $GL(V)$ preserving $h$ is the symplectic group $Sp(V)$,
so we have $U(V)=Sp(V)$ in this case. Now $V$ has a symplectic basis $u_1,v_1,\dots,u_m,v_m$ relative to which the Gram
matrix is $J_m=\begin{pmatrix} 0 & 1_m \\ -1_m & 0 \end{pmatrix}$. For \( m=1 \), an easy calculation shows that the group preserving
this $J_1$ is $SL(2,R)$, that is, $Sp(2,R)=SL(2,R)$. Thus, if $m=1$ we have $U(V)=Sp(V)=SU(V)$. But then the inductive proof of Theorem
\ref{thm:SUgenbytransvections} applies to yield the following result.

\begin{theorem}\label{symp} Suppose the involution is trivial $*$. Then $U(V)=Sp(V)=SU(V)$ is generated by
the set of all symplectic transvections $f_{a,v}$, where $a\in R$ and $v\in V$ is a basis vector. In particular,
every symplectic transformation of $V$ has determinant 1.
\end{theorem}

Going back to our general set up, we observe that our arguments also allow us to obtain an
elementary proof of Witt's extension theorem
in this context.
This result is already known in greater generality
(see \cite{R} and \cite{F}) but we include it here
as our arguments give an elementary and self contained proof in
our particular setting.
More precisely, we have the following.

\begin{theorem}
    Suppose that \( w_1,z_1,\dots,w_k,z_k \) is a
    symplectic set in \( V \) for some \( k \leq m \).
    Then there exist \( w_{k+1},z_{k+1},\dots,w_m,z_m \)
    such that \( w_1,z_1,\dots,w_m,z_m \) is a symplectic basis
    for \( V \).
    \label{thm:WittCancellation}
\end{theorem}

\begin{proof}
    We prove this by induction on \( m \). The case
    \( m=1 \) is trivial as either \( k=0 \) or \( k=1 \)
    in this case.
    For the inductive step, we note that by Corollary
    \ref{cor:transitivityonsymppairs} there is some
    \( X \in SU(V) \) such that \( Xw_1 = u_1 \) and
    \( Xz_1 = v_1 \). Thus we can assume, without loss of
    generality, that \( w_1 = u_1 \) and \( z_1 = v_1 \).
    Now \( w_2,z_2,\dots,w_k,z_k \) is a
    symplectic set in \( (Sw_1+Sz_1)^\perp = (Su_1+Sv_1)^\perp
    = Su_2+Sv_2+\dots+Su_m+Sv_m\).
    By induction \( w_2,z_2,\dots,w_k,z_k \)
    can be extended to a symplectic basis of
\( Su_2+Sv_2 + \dots + Su_m+Sv_m \) as required.
\end{proof}

\section{Surjectivity of reduction homomorphisms}

We maintain all hypotheses and notation adopted in \S\ref{antes}. In addition,
we assume in this section that $2\in S^\times $ and set \( \tilde S = S/\mathfrak i \),
a quotient ring of \( S \) where \( \mathfrak i \) is proper \( * \)-invariant
ideal of \( S \). Clearly \(\tilde  S \) is also a local ring
and \( * \) induces an involution on \( \tilde S \) that we
also denote by \( * \).
Moreover \( \tilde V = V/\mathfrak i V \) is a free \( \tilde S \)
module of rank \( 2m \).
For \( v \in V \) let \( \tilde v \) be the image of \( v \) under
the quotient map \( V \rightarrow \tilde V \).
The
composition
\( V \times V \rightarrow S \rightarrow \tilde S \)
factors through the quotient \( V \times V \rightarrow \tilde V \times \tilde V \)
and
induces a skew hermitian form on \( \tilde V \), which we denote
by \( \tilde h \). It is clear that
\( \tilde u_1,\tilde v_1,\dots,\tilde u_m,\tilde v_m \)
is a symplectic basis for \( \tilde V \).
Given \( X \in U(V) \), define \( \tilde X \) by
\( \tilde X \tilde u = \widetilde{Xu}  \).
Clearly the mapping \( X \mapsto \tilde X \)
defines a homomorphism
\( \pi:SU(V) \rightarrow SU(\tilde V) \), which we will refer to as
the reduction homomorphism.
In this section we will prove that this homomorphism is surjective.

\begin{lemma}
    \label{lem:liftingbasisvectorsoflengthzero}
    Let \( \tilde u \) be a basis vector of \( \tilde V \) of
    length zero. There is some \( w \in V \) such that
    \( w \) is a basis vector of \( V \) of length \( 0 \)
    and \( \tilde w = \tilde u \).
\end{lemma}

\begin{proof}
    Since \( \tilde u \) is a basis vector of \( \tilde V \), Lemma \ref{p} allows us to
    assume, without loss of generality, that
    \( \tilde h (\tilde u, \tilde v_1) \in \tilde S^\times\).
    Therefore \( a = h(u,v_1) \in S^\times \).
    By assumption, we have \( \delta = h(u,u) \in \mathfrak i \).
    Let \( w  = u - \frac12 a^{-1}\delta v_1 \).
    Clearly \( w \) is a basis vector since \( u \) is a
    basis vector.
    Moreover \( \tilde w = \tilde u \) and one readily checks that
    \( h(w,w) = 0 \) as required.
\end{proof}

\begin{theorem} If $2\in S^\times$ then
    the reduction homomorphism \( \pi:SU(V) \rightarrow
    SU(\tilde V) \) is surjective.
    \label{thm:surjectivityofreduction}
\end{theorem}

\begin{proof}
    By Theorem \ref{thm:SUgenbytransvections} it suffices
    to show that \( f_{\tilde a, \tilde u} \) is in the
    image of \( \pi \) for all \( \tilde a \in \tilde S \)
    such that \( \tilde a^* = \tilde a \) and all
    basis vectors
    \( \tilde u  \) in \( \tilde V \) of length zero.
    By Lemma \ref{lem:liftingbasisvectorsoflengthzero}
    we can assume that \( h(u,u) = 0 \). Also, since
    \( \widetilde{\frac{a^*+a}2} = \tilde a \), we can
    assume that \( a^* = a \). Now \( f_{a,u}
    \in SU(V)\) and it is clear that \( \pi(f_{a,u})
    =f_{\tilde a,\tilde u} \).
\end{proof}

When the involution $*$ is trivial all vectors in $V$
have length 0 and all elements of $S$ are fixed by $*$. Thus, the
condition that $2\in S^\times$ is not required and as a
consequence of Theorem \ref{symp} we have the following result.

\begin{theorem} The reduction homomorphism \( \pi:Sp(V) \rightarrow
    Sp(\tilde V) \) is surjective.
\end{theorem}

Next we consider the reduction homomorphism \( \pi:
U(V) \rightarrow U(\tilde V)\). Recall that \( a^*a \)
is the norm of \( a \). Let \( N \)
be the group of elements of norm
one in \( S \). Clearly \( \det(X) \in N \) for
\( X \in U(V) \). We would like to identify the
image of \( \det \). There are two cases to consider.

Let $\mathfrak j$ stand for the maximal ideal of $S$. We say that \( * \) is a ramified involution of \( S \)
if \( a^* - a \in \mathfrak j \) for all \( a \in S \).
Otherwise, we say that \( * \) is unramified. Observe
that \( * \) is unramified if and only if there is
some \( i \in S^\times \) such that \( i^* = -i \).

\begin{lemma}
    Suppose that \( * \) is unramified or that
    \( a \not\equiv -1 \mod \mathfrak j \).
    Then \( a \in N \) if and only if there is some
    \( b \in S^\times \) such that \( a = b(b^*)^{-1} \).
    \label{lem:normonedecompunramified}
\end{lemma}

\begin{proof} It is obvious that $aa^*=1$ provided \( a = b(b^*)^{-1} \) for some \( b \in S^\times \).
Suppose, conversely, that $aa^*=1$. Assume first that
    \( a \not\equiv -1 \mod \mathfrak j \).
    Then \( b = 1+a \in S^\times \) and therefore
    \( b^* \in S^\times \) also.
    Now
    \begin{eqnarray*}
        b(b^*)^{-1} &=& \frac{1+a}{1+a^*} \\
        &=& \frac{(1+a)^2}{(1+a^*)(1+a)} \\
        &=& \frac{1+2a+a^2}{1+a+a^*+a^*a} \\
        &=& \frac{a^*a+2a+a^2}{2+a+a^*} \\
        &=& a
    \end{eqnarray*}
    Assume next that \( a \equiv -1 \mod \mathfrak j \).
    Then $-a\equiv 1\not\equiv -1\mod \mathfrak j$, so by above \( c = 1-a \in S^\times \)
    satisfies $c(c^*)^{-1} = -a$. On the other hand, by assumption \( * \) is unramified and as observed above
    there is some \( i \in S^\times \) such that \( i^* = -i \), so that \( i(i^*)^{-1} = -1 \). It follows that $b=i(1-a)\in S^\times$
    satisfies \( b(b^*)^{-1} = a \).
\end{proof}

Observe that in the ramified case \( a^*a = 1 \) implies that
\( a^2 \equiv 1 \mod \mathfrak j \). So \( a \equiv \pm 1 \mod \mathfrak j \).
Now, if \( a \equiv -1 \mod \mathfrak j \) then there is no
\( b \) such that \( b(b^*)^{-1} = a \), for if there were such \( b \), then
\( 1\equiv bb^{-1} \equiv -1 \mod \mathfrak j\) which contradicts \( 2 \in S^\times \).

\begin{lemma}\label{q}
    Suppose that \( * \) is unramified. Then \( \det: U(V) \rightarrow
    N\) is a surjective group homomorphism.
    \label{lem:surjectivedetunramified}
\end{lemma}

\begin{proof}
    Let \( a \in N \). By Lemma \ref{lem:normonedecompunramified}
    \( a = b(b^*)^{-1} \) for some \( b \in S^\times \).
    Let \[ X = diag(b,(b^*)^{-1},1,\dots,1) .\] Clearly
    \( \det (X) = a \) and one readily checks that \( X^*JX = J \)
    as required.
\end{proof}

\begin{lemma}
    Suppose that \( * \) is ramified. Then \( \text{im}(\det)
    = N \cap (1+\mathfrak j)\).
    \label{lem:imagedetramified}
\end{lemma}

\begin{proof}
    If \( a \in N \cap (1+\mathfrak j) \) then by Lemma
    \ref{lem:normonedecompunramified} \( a = b(b^*)^{-1} \)
    and then proceeding as in the proof of Lemma \ref{q}
    we see that \( a \in \text{im}(\det) \).
    On the other hand if \( X \in U(V)  \), then its reduction
    modulo \( \mathfrak j \) is a symplectic matrix over \( S/\mathfrak j \).
    It is well-known that a symplectic matrix over a field has determinant one
    (see Theorem \ref{symp}). Therefore \( \det X \equiv 1 \mod \mathfrak j \)
    as required.
\end{proof}

\begin{theorem} If $2\in S^\times$ then
    the reduction homomorphism \( \pi:U(V) \rightarrow U(\tilde V) \)
    is surjective.
    \label{thm:surjectivityofreductionfullunitary}
\end{theorem}

\begin{proof}
    Let \( Z \in U(\tilde V)  \) and let \( \tilde a = \det (Z) \) for some \( a \in S \).
    Suppose first that \( * \) is an unramified involution of \(\tilde S\).
    By Lemma \ref{lem:normonedecompunramified}
    there is some \( b \in S^\times \) such that \( \tilde b (\tilde b ^*)^{-1} = \tilde a \).
    Let \( c = b(b^*)^{-1} \). Now \( \tilde c = \tilde a \). Let
    \( X = diag(b,(b^*)^{-1},1,\dots,1) \in M(2m,S) \) and observe that \( X \in U(V) \) and
    that \( \det (\tilde X) = \det (Z) \). Therefore \( Z\tilde X^{-1} \in SU(\tilde V) \)
    and by Theorem~\ref{thm:surjectivityofreduction} there is some \( Y \in SU(V) \)
    such that \( Z\tilde X^{-1} = \tilde Y \). Therefore \( Z = \widetilde{YX} \) as required.

    Now suppose that \( * \) is ramified. Then by Lemma \ref{lem:imagedetramified},
    \( a \equiv 1 \mod \mathfrak j \). Once again by Lemma \ref{lem:normonedecompunramified}
    there is some \( b \in S^\times \) such that \( \tilde b (\tilde b ^*)^{-1} = \tilde a \).
    Now proceeding exactly as in the previous case, we again conclude that \( Z \in \text{im}(\pi) \).
\end{proof}

\noindent{\bf Acknowledgment.} We thank the referee for a careful reading of the manuscript.



\begin{thebibliography}{RBMW}

\bibitem[ADW]{ADW} J. Ahrens, A. Dress and H. Wolf, \emph{Relationen zwischen Symmetrien in orthogonalen Gruppen},
J. Reine Angew. Math. 234 (1969) 1--11.

\bibitem[B]{B} R. Baeza, \emph{Eine Zerlegung der unit$\mathrm{\ddot{a}}$ren Gruppe $\mathrm{\ddot{u}}$ber lokalen Ringen}, Arch. Math. 24 (1973) 144-–157.


\bibitem[Bo]{Bo} B$\mathrm{\ddot{o}}$ge, \emph{Definierende Relationen zwischen Erzeugenden der klassischen Gruppen},
Abh. Math. Sem. Univ. Hamburg 30 (1967) 165-–178.

\bibitem[CG]{CG} H.L. Claasen and R.W. Goldbach, \emph{A field-like property
of finite rings}, Indag. Math. 3 (1992) 11-–26.

\bibitem[CQS]{CQS} J. Cruickshank, R. Quinlan and F. Szechtman, \emph{Hermitian and skew hermitian forms over local rings},
 arXiv:1705.01562.

\bibitem[D]{D} J. Dieudonn$\mathrm{\acute{e}}$, \emph{La g$\mathrm{\acute{e}}$om$\mathrm{\acute{e}}$trie des groupes classiques}, 3$\mathrm{\grave{e}}me$ ed., Springer-Verlag, Berlin, 1971.

\bibitem[E]{E} Ellers, Relations in classical groups, J. Algebra 51 (1978) 19-–24.

 \bibitem[F]{F} U. A. First, \emph{Witt's Extension Theorem for Quadratic Spaces over Semiperfect Rings}, J. Pure Appl. Algebra 219 (2015) 5673--5696.

\bibitem[GHH]{GHH} S. Gurevich, R. Hadani and R. Howe,
\emph{Quadratic reciprocity and the sign of the Gauss sum via the finite Weil representation},
Int. Math. Res. Not. IMRN 19 (2010) 3729-–3745.

\bibitem[GV]{GV} L. Guti\'errez Frez and A. Vera-Gajardo, \emph{Weil representations of $\mathrm{U}(n, n)(F_{q^2}/F_q)$ , $q > 3$ odd, via presentation and compatibility of methods}, Comm. Algebra 46 (2018) 653--663.

\bibitem[G]{G} G$\mathrm{\ddot{o}}$tzki,\emph{Unverk$\mathrm{\ddot{u}}$rzbare Produkte und Relationen in unitaren Gruppen},
Math. Z. 104 (1968) 1--15.

\bibitem[HO]{HO} A. Hahn and O.T. O'Meara, \emph{The classical groups and $K$-theory}, Springer-Verlag, Berlin, 1989.


\bibitem[H]{H} I. Herstein, \emph{Noncommutative rings}, The Carus Mathematical
Monographs, no. 15, Mathematical Association of America, 1968.

\bibitem[Ho]{Ho} T. Honold, \emph{Characterization of finite Frobenius
rings}, Arch. Math. (Basel)  76  (2001) 406-–415.

\bibitem[K]{K} W. Klingenberg, \emph{Symplectic groups over local rings},  Amer. J. Math. 85 (1963) 232--240.

\bibitem[J]{J} N. Jacobson, \emph{Basic Algebra I}, W.H. Freeman, New York, 1985.


\bibitem[La]{La} E. Lamprecht, \emph{\"{U}ber $I$-regul\"{a}re Ringe, regul\"{a}re Ideale und
Erkl\"{a}rungsmoduln. I.}, Math. Nachr. 10 (1953) 353-–382.

\bibitem[PS]{PS} J. Pantoja and J. Soto-Andrade,
\emph{A Bruhat decomposition of the group ${\rm SL}_*(2,A)$},
J. Algebra 262 (2003) 401-–412.

\bibitem[Pr]{Pr} D. Prasad, \emph{A brief survey on the theta correspondence}, pp. 171–193 in Number
theory (Tiruchirapalli, 1996), edited by V. K. Murty and M. Waldschmidt, Contemp. Math. 210,
Amer. Math. Soc., Providence, RI, 1998.



\bibitem[P]{P} J. Pantoja, \emph{A presentation of the group ${\rm Sl}_\ast(2,A)$, $A$ a simple
   Artinian ring with involution}, Manuscripta Math. 121 (2006) 97--104.

\bibitem[R]{R} H. Reiter, \emph{Witt's theorem for noncommutative semilocal rings}, J. Algebra 35 (1975) 483--499.

\bibitem[S]{S} F. Szechtman, \emph{Quadratic Gauss sums over
finite commutative rings}, J. Number Theory 95 (2002) 1-–13.

\bibitem[V]{V} Vera-Gajardo. \emph{A generalized Weil representation for the finite split orthogonal group $\mathrm{0}_q(2n,2n)$, $q$ odd $>3$},
J. Lie Theory 25 (2015) 257--270.

\bibitem[Wo]{Wo} J. Wood, \emph{Duality for modules over finite rings and
applications to coding theory}, Amer. J. Math. 121 (1999)
555-–575.





\end{thebibliography}
\end{document}